\documentclass[12pt,a4paper]{amsart}
\usepackage{tikz}
\usepackage{tikz-cd}
\usetikzlibrary{decorations.markings,decorations.pathmorphing,cd}
\usetikzlibrary{arrows}
\usetikzlibrary{calc}
\usepackage{pgfplots}
\pgfplotsset{every axis/.append style={
                    axis x line=middle,    
                    axis y line=middle,    
                    axis line style={-,color=blue}, 
                    xlabel={$x$},          
                    ylabel={$y$},          
            }}
\usepackage[latin1]{inputenc}
\usepackage{amsmath}
\usepackage{amsthm}
\usepackage{amsfonts}
\usepackage{amssymb}
\usepackage{mathtools}
\usepackage{enumerate}
\usepackage{enumitem}
\usepackage{graphicx}
\usepackage{hyperref}
\usepackage{nicefrac}
\usepackage{cancel}

\DeclareMathOperator{\Hir}{Hir}

\DeclareMathOperator{\gen}{gen}
\DeclareMathOperator{\mult}{mult}

\DeclareMathOperator{\Spec}{Spec}

\def\cS{{\mathcal S}}

\def\ZZ{\mathbb{Z}}

\def\CC{\mathbb{C}}

\def\AA{\mathbb{A}}
\def\EE{\mathbb{E}}

\def\QQ{\mathbb{Q}}

\def\RR{\mathbb{R}}

\def\cO{\mathcal{O}}

\def\cF{\mathcal{F}}

\def\div{{\rm div}}

\def\val{\mathrm{val}}

\makeatletter
\let\c@lemma\c@thm
\let\c@prop\c@thm
\let\c@propdef\c@thm
\let\c@proper\c@thm
\let\c@problem\c@thm
\let\c@conj\c@thm
\let\c@cor\c@thm
\let\c@rem\c@thm
\let\c@dfn\c@thm
\let\c@notation\c@thm
\let\c@exam\c@thm
\makeatother

\newcommand{\m}{\mathfrak{m}}
\newcommand{\tX}{\widetilde{X}}

\numberwithin{equation}{section}
\numberwithin{equation}{subsection}

\theoremstyle{plain}

\newtheorem{theorem}[equation]{Theorem}
\newtheorem{lemma}[equation]{Lemma}

\newtheorem{prop}[equation]{Proposition}

\theoremstyle{definition}

\newtheorem{example}[equation]{Example}

\newtheorem{remark}[equation]{Remark}

\title
{Ordinary $r$--tuples in cyclic quotient surface singularities}
\date{\today}

\author[J.I. Cogolludo]{Jos{\'e} I. Cogolludo-Agust{\'i}n}
\address{Departamento de Matem\'aticas, IUMA\\
Universidad de Zaragoza\\
C.~Pedro Cerbuna 12\\
50009 Zaragoza, Spain}
\email{jicogo@unizar.es}

\author[T. L\'aszl\'o]{Tam\'as L\'aszl\'o}
\address{Babe\c{s}-Bolyai University, Str. Mihail Kog\u{a}lniceanu nr. 1, 400084 Cluj-Napoca, Romania}
\email{tamas.laszlo@ubbcluj.ro}

\author[J. Mart\'in]{Jorge Mart\'in-Morales}
\address{Departamento de Matem\'aticas, IUMA\\
Universidad de Zaragoza\\
C.~Pedro Cerbuna 12\\
50009 Zaragoza, Spain}
\email{jorge.martin@unizar.es}

\author[A. N\'emethi]{Andr\'as N\'emethi}
\thanks{AN and TL are  partially supported by ``\'Elvonal (Frontier)'' Grant KKP 144148.
TL is partially supported by the  J\'anos Bolyai Research Scholarship of the Hungarian Academy of Sciences.
JM is supported by MCIN/AEI/10.13039/501100011033 and the European Union \emph{NextGenerationEU/PRTR} (grant
code: RYC2021-034300-I).
JC and JM are partially supported by PID2020-114750GB-C31, funded by
MCIN/AEI/10.13039/501100011033 and also partially funded by the Departamento de Ciencia,
Universidad y Sociedad del Conocimiento of the Gobierno de Arag\'on
(Grupo de referencia E22\_20R ``\'Algebra y Geometr\'{\i}a'').
JM is also supported by Junta de Andaluc\'ia (FQM-333).}

\address{HUN-REN Alfr\'ed R\'enyi Institute of Mathematics,
Re\'altanoda utca 13-15, H-1053, Budapest, Hungary \newline
 \hspace*{4mm} ELTE - University of Budapest, Dept. of Geometry, Budapest, Hungary \newline \hspace*{4mm}
  BBU - Babe\c{s}-Bolyai Univ., Str, M. Kog\u{a}lniceanu 1, 400084 Cluj-Napoca, Romania
 \newline \hspace*{4mm}
BCAM - Basque Center for Applied Math.,
Mazarredo, 14 E48009 Bilbao, Basque Country, Spain}
\email{nemethi.andras@renyi.hu }

\subjclass[2010]{Primary. 14B05, 32Sxx; Secondary. 14E15}
\keywords{Normal surface singularities, delta invariant of curves, rational surface singularities, cyclic quotient, Weil divisors, $r$--tuples}

\begin{document}

\begin{abstract}
We describe those Weil divisors of cyclic quotient surface singularities which are (abstract) $r$--tuple curve singularities.
\end{abstract}

\maketitle

\begin{quotation}
\begin{center}
\emph{
Dedicated to our friends Alejandro Melle and Enrique Artal.
}
\end{center}
\end{quotation}

\section{Introduction}

\subsection{}
The delta invariant $\delta(C,o)$ of a complex analytic reduced curve germ $(C,o)$
with $r$ local irreducible components satisfies $\delta(C,o)\geq r-1$. The optimal minimum
$r-1$ is realized by  {\it ordinary $r$--tuples}.  These germs, by definition,
are analytically isomorphic with the union of the coordinate axes of $(\CC^r,o)$ (for more details see section 2).

Our discussion targets the following problem: given a complex analytic  normal surface
singularity $(X,o)$, when is possible to embed an ordinary $r$--tuple $(C,o)$ in $(X,o)$
as a reduced Weil divisor. There are several immediate obstructions: for example, 
$r$ cannot be strict larger than the embedding dimension ${\rm embdim}(X,o)$ of $(X,o)$.
In some specific cases (see below)
this optimal bound is realized. But usually (and more frequently)
even the case $r=1$ is prohibited: $(X,o)$ contains no smooth curve germ.
For a discussion and more particular cases see section~\ref{sec:the-guiding-question}.

One can ask an even more complex question: if we fix a resolution $\widetilde{X}$ of $(X,o)$ (here, in fact, we can assume that $\widetilde{X}$ is minimal), and we consider the
irreducible exceptional curves $\{E_v\}_{v=1}^s$, then we can ask about the possible distribution of the intersection points of the strict transforms of an ordinary $r$--tuple $(C,o)\subset (X,o)$ with different components $E_v$.

In this note we will treat the case of a cyclic quotient singularity $(X,o)$. In this case
it is a well--known  that the optimal bound
$r={\rm embdim}(X,o)$ is realized. E.g., if we take a generic linear function on $(X,o)$, then its zero locus is an $({\rm embdim}(X,o)-1)$--tuple. Even more, if we add to this
curve another component, the projection of a smooth generic transversal cut in $\widetilde{X}$ intersecting an arbitrarily chosen $E_v$, then this union is an
${\rm embdim}(X,o)$--tuple. On the other hand, in \cite[Example 5.9]{kappa1} we found ordinary $r$--tuples  which are not parts of an ${\rm embdim}(X,o)$--tuple constructed
in this way.

These facts motivated us to search for a common generalization: to characterize those
generic transversal cuts of the exceptional irreducible components whose union projects 
into an  $r$--tuple.

The main result of the note is the following.

\begin{theorem} Fix the minimal resolution of a cyclic quotient singularity  $(X,o)$.
Assume that on each $E_v$  we put $r_v$ different transversal discs in $\widetilde{X}$
$(1\leq v\leq s)$. Then the projection
into $(X,o)$ forms a $(\sum_v r_v)$--tuple whenever the following inequality holds
for any $1\leq v_1\leq v_2\leq s$
\begin{equation*}
\sum_{v_1\leq v\leq v_2}r_v \leq
1 + \sum_{v_1\leq v\leq v_2} ( k_v-{\rm val}_v ).
\end{equation*}
Here $-k_v$ denotes the self-intersection of $E_v$ in $\widetilde{X}$,  and $\val_v$ is the valency of $v$.
\end{theorem}

\section{Preliminary. The delta invariant of curves and $r$--tuples}

\subsection{} In this section we review some facts  about the delta invariant of a complex analytic curve germ,
which are relevant from the point of view of the present note. In some parts
we follow~\cite{BG80} and~\cite{StevensThesis}.

Let $(C,o)$ be the germ of a complex reduced curve singularity, let $n:(C,o)^{\widetilde{}}\to(C,o)$
be the normalization, where $(C,o)^{\widetilde{}}$ is the multigerm $(\widetilde{C},n^{-1}(o))$.
The delta invariant is defined as $\dim_{{\CC}}n_*\cO_{(C,o)^{\widetilde{}}}/\cO_{(C,o)}$.
We also write $r$ for the number of irreducible components of $(C,o)$.
The delta invariant of a reduced curve can be determined inductively from the delta invariant of
the components and the \emph{Hironaka generalized intersection multiplicity}.
Indeed, assume $(C,o)$ is embedded in some $(\CC^n,o)$, and assume that $(C,o)$ is the union of
two (not necessarily irreducible) germs $(C_1',o)$ and $(C_2',o)$ without common irreducible components.
Assume that $(C_i',o)$ is defined by the ideal $I_i$ in $\cO_{(\CC^n,o)}$ ($i=1,2$).
Then one can define \emph{Hironaka's intersection multiplicity} by
$(C_1',C_2')_{\Hir}:=\dim _{\CC}\, (\cO_{(\CC^n,o)}/ I_1+I_2)$.
The following formula is due to Hironaka, see~\cite{Hironaka} or~\cite[2.1]{StevensThesis} and~\cite{BG80},
\begin{equation}\label{eq:deltaA1}
\delta(C,o)=\delta(C_1',o)+\delta(C_2',o)+(C_1',C_2')_{\Hir}.
\end{equation}
In particular, if $(C,o)$ has a decomposition in $r$ irreducible components,
then using  induction
 (and  $(C_1',C_2')_{\Hir}\geq 1$) one obtains
\begin{equation}\label{eq:deltaA2}
\delta(C,o)\geq r-1.\end{equation}

\subsection{Ordinary $r$--tuples} \label{ex:delta22} \cite{BG80,StevensThesis}
If a reduced curve  germ $(C,o)$ is analytically equivalent with the union of the coordinate
axes of $(\CC^r,o)$, then we call $(C,o)$ {\it ordinary $r$--tuple}. Using e.g.
Hironaka's formula we get that $\delta(C,o)=r-1$ for any $r$--tuple. This shows that ordinary $r$--tuples realize
the optimal minimum of the bound~\eqref{eq:deltaA2}. In fact, this property characterizes them:
if a curve $(C,o)$ with $r$ components satisfies $\delta(C,o)= r-1$, then $(C,o)$ is necessarily an
ordinary $r$--tuple.

\section{The guiding question}\label{sec:the-guiding-question}

\subsection{The question and first examples.}\label{ss:first} \

Let $(X,o)$ be a normal surface singularity, and fix a Weil
divisor $(C,o)$ of $(X,o)$. We wish to list cases when $(C,o)$ is an ordinary $r$--tuple. In other words,
we might ask the following question: when is it possible to embed an ordinary $r$--tuple as a Weil divisor
in a normal surface singularity? We start the discussion with some concrete examples.

First, let us consider some immediate obstructions. Let us fix a Weil divisor $(C,o)$ of a normal surface singularity
$(X,o)$, and also an embedded resolution $\pi:\widetilde{X}\to X$ of the pair $(C,o)\subset (X,o)$. Let $E:=\pi^{-1}(o)$
be the exceptional curve. Let $\{\widetilde{C}_i\}_{i=1}^r$ denote the strict transforms of the irreducible components
$\{C_i\}_{i=1}^r$ of~$(C,o)$.

If $(C,o)$ is an ordinary $r$--tuple then each $(C_i,0)$ is smooth, hence it has local multiplicity one at $o$.
That is, the generic linear function intersects $(C_i,o)$ at $o$ with multiplicity one. This can be tested on the
embedded resolution as well. Indeed, let $E_v$ be the irreducible exceptional divisor which intersects $\widetilde{C}_i$
and write $p_i:=\widetilde{C}_i\cap E_v$. Let $\pi^*(\mathfrak{m}_{(X,o)})$ be the pullback of the maximal ideal
$\mathfrak{m}_{(X,o)}$ of $\cO_{(X,o)}$ considered as an $\cO_{\widetilde{X}}$--module.
Then $\pi^*(\mathfrak{m}_{(X,o)})={\mathcal B}\cdot \cO_{\widetilde{X}}(-Z_{\max})$ for an ideal sheaf ${\mathcal B}$
of $\cO_{\widetilde{X}}$ supported in finitely many points (they sit on the exceptional curve $E$, and they are called
the \emph{base points} of $\pi^*(\mathfrak{m}_{(X,o)})$) and the divisorial part $\cO_{\widetilde{X}}(-Z_{\max})$.
The cycle $Z_{\max}$, usually called the maximal ideal cycle, is that part of the  divisor of the generic linear function
which is supported on $E$. If we denote by $Z_{\min}$ the minimal (Artin) fundamental cycle associated with the resolution,
then $Z_{\max}\geq Z_{\min}$.

Now, if $(C_i,o)$ is smooth, then (i) the $E_v$--coefficient of $Z_{\max}$ should be one and (ii) $p_i$ cannot be in
the support of ${\mathcal B}$. In general it is not very easy to determine the cycle $Z_{\max}$. But, for instance, if
the $E_v$--coefficient of the topological cycle $Z_{\min}$ is already larger than one, then $E_v$ cannot support the
strict transform of a smooth $(C_i,o)$.
Moreover, such properties can be tested in the minimal resolution (see e.g. Remark
\ref{rem:utolso}).
For example, the $\EE_8$ singularity contains no smooth irreducible Weil divisor.

On the other hand, condition (ii) is also restrictive. Consider for instance the hypersurface singularity
$(X,o)=\{x^2+y^3+z^7=0\}$ in $(\mathbb{C}^3,o)$ and its minimal good resolution. The strict transform of the Weil
divisor $(C,o)=\{z=0,\ x^2+y^3=0\}$ intersects transversally the $(-7)$ irreducible exceptional divisor
(say $E_v$) at a smooth point $p$ of $E$. In this case $Z_{\max}=Z_{\min}$ and the $E_v$--coefficient of $Z_{\min}$ is one.
(In fact, it is the only irreducible exceptional divisor whose coefficient in $Z_{\min}$ is one.)
However, the point $p$ is a base point of $\pi^*(\mathfrak{m}_{(X,o)})$, which is in agreement with the fact that
$(C,o)$ is not smooth.
On the other hand, any generic transversal cut (i.e. a smooth curve germ intersecting $E_v$ transversally)
$\widetilde{C}$ of $E_v$ (with $p\not\in \widetilde{C}$) projected via $\pi$ gives a smooth Weil divisor of~$(X,o)$.

Next, regarding this last example, one can ask: if we put two  transversal cuts $\widetilde{C}_1$ and $\widetilde{C}_2$
on $E_v$, and we project them  via $\pi$, is this curve $C_1\cup C_2$ an ordinary 2--tuple? That is, does the delta
invariant of this pair $C_1\cup C_2$ (with both components smooth) equal 1? The answer is no, this delta invariant is~2.

This can be seen (at least) in two different ways. The first one uses equations of the curve.
For example, $(C_1,o)$ can be given e.g. by $y=-z^2$ and $x=z^3\sqrt{1-z}=z^3(1-z/2-z^2/8+\cdots$, while
$(C_2,o)$ by $y=z^2$ and $x=iz^3\sqrt{1+z}=iz^3(1+z/2-z^2/8+\cdots$. (The fact that the strict transforms have the
required properties can be seen if we consider the weighted blow up with weights $(3,2,1)$
which creates $E_v$.) Then  $I_1+I_2=(x,y,z^2)$, hence $(C_1,C_2)_{\Hir}=2$.

Another possibility (which does not use concrete equations, which usually are hard to handle) is the following:
though the function $z$ has multiplicity 1 along $E_v$ the multiplicities of $x$ and $y$ are 3 and 2. So, both
curves  $C_1$ and $C_2$ embed into the three--space (via the  coordinates $x,y,z$) by equations of type
$x=a_it^3, \ y=b_it^2, \ z=c_it$ (with $a_i, b_i,c_i$ nonzero constants, $i=1,2$), hence they have common
tangent lines, which is not the case for an ordinary 2--tuple.

Both arguments hold because there is `only one function' (up to a nonzero constant)
which has multiplicity one along~$E_v$.

Let us analyze the case of $\AA_4$ with equation $(X,o)=\{z^5=xy\}$ as well. The minimal resolution has
4 exceptional irreducible divisors, all $(-2)$ curves, say $E_1,\ldots, E_4$. In this case $Z_{\min}=Z_{\max}=E$
and the maximal ideal has no base points. So any transversal cut on any of the $E_v$'s projects into a smooth curve.
But by the same argument as above, on the very same $E_v$ with $v=2$ or $v=3$,
one cannot put two transversal curves which projects on an ordinary 2--tuple. Furthermore,
a transversal cut on $E_2$ and a cut of $E_3$ project into a curve with two smooth components but with $\delta=2$.
Indeed, analyzing the multiplicities of all the coordinates,
one of them has parametrization of type  $x=a_1t^3, \ y=b_1t^2, \ z=c_1t$, the other
$x=a_2t^2, \ y=b_2t^3, \ z=c_2t$, hence, indeed, $(C_1,C_2)_{\Hir}=2$.
On the other hand, a transversal cut of $E_1$ and a cut of $E_4$ project on an ordinary 2--tuple
(it can be realized e.g. by $z=xy=0$).

Similarly, on $\AA_k$ with exceptional curves $E_1,\ldots, E_k$, one cut on any $E_v$ projects onto a smooth curve,
but a pair of cuts supported on $E_2, \ldots, E_{k-1}$ projects onto a curve with~$\delta\geq 2$.

Finally, note also that $(C,o)\subset (X,o)$ implies ${\rm embdim}(C,o)\leq {\rm embdim}(X,o)$.
In particular, if $(C,o)$ is an ordinary $r$--tuple then $r\leq {\rm embdim}(X,o)$.

In the case of $\AA_k$, for which ${\rm embdim}(X,o)=3$, the maximal value $r=3$ can be realized as an
ordinary $r$--tuple by a curve which is a projection of a transversal cut at $E_1$, $E_k$ and an arbitrary $E_v$.
However, for $\EE_8$, ${\rm embdim}(X,o)=3$ but $(X,o)$ contains no smooth curves as mentioned above.

\subsection{The generic linear section on minimal rational singularities}\label{s:embcurves}\

This section follows the general ideas about generic linear sections on minimal rational singularities and
ordinary $r$--tuples introduced in~\cite{StevensThesis} (see also \cite[Examples 5.2, 5.3]{kappa1}).

\subsubsection{} \label{bek:rat1}
Let us fix a rational surface singularity $(X,o)$ and its minimal resolution $\pi:{\tilde X}\to X$. We will
assume that the fundamental cycle $Z_{\min}$ is reduced:  $Z_{\min}=E$. (Such singularities are called
`rational minimal' in \cite{KSB}, they coincide also with the family of `rational Kulikov singularities'
\cite{Karras80,StevensThesis}.)

Furthermore, by \cite{Artin66}, the multiplicity $\mult(X,o)$ of $(X,o)$ is $-Z_{\min}^2$ and the embedded
dimension of $(X,o)$ is $-Z_{\min}^2+1$. Set $r:=-Z_{\min}^2$, and fix an embedding
$(X,o)\subset (\mathbb{C}^{r+1},o)$. Let $\ell$ be the generic linear function of $(\CC^{r+1},0)$, it induces
the `generic linear section'
of $(X,o)$. Set $(C,o):=\{\ell=0\}\cap (X,o)$.
It is well known (see e.g.~\cite{StevensThesis}) that \emph{$(C,o)$ is an ordinary $r$-tuple}.

Indeed, $(C,o)$ is embedded in
the linear hyperplane $\{\ell=0\}$ of $(\CC^{r+1},o)$, hence in some $\CC^r$. Furthermore, since the maximal
ideal of $\cO_{(X,o)}$ has no base point (in the minimal resolution), we  obtain that
$\widetilde{C}\subset {\tilde X}$ is a union of smooth discs transversal to $E$, and their number is
$r=-Z_{\min}^2=(\widetilde{C},Z_{\min})=(\widetilde{C},E)$.
Since $E=Z_{\min}$ intersects each disc $\widetilde{C}_i$ with multiplicity one, the corresponding
component $C_i=\pi(\widetilde{C}_i)$ of $(C,o)$ intersects $\ell$ with multiplicity one, hence it is smooth.
Furthermore, by the base point freeness, all the $r$ smooth components of $(C,o)$ in
$\{\ell=0\}\subset (\CC^r,0)$ are in general position. (One can also argue inductively as follows. Assume that
$\{\ell =0\}\cap (X,o)=\cup _{i=1}^r C_i$ as above and we already know that $\cup_{i=1}^kC_i$ is an ordinary
$k$--tuple. For any $i>k$ move $\widetilde{C}_i$ generically into $\widetilde{C}_i'$. Then
$\cup_{i=1}^kC_i\cup\cup_{i>k}C_i'$ is also cut out by a linear function, say $\ell'$. This also contains
$\cup_{i=1}^kC_i$, but intersects $C_{k+1}$ with multiplicity one. Hence, by Hironaka's formula and \ref{ex:delta22},
$\cup_{i=1}^{k+1}C_i$ is an ordinary $(k+1)$--tuple.) Hence $(C,o)$ is an ordinary $r$--tuple.
(For more see e.g.~\cite{StevensThesis}.)

\subsubsection{}\label{bek:rat2}
In fact, we can prove even more, cf. \cite[Example 5.3]{kappa1}. Let $\widetilde{D}^+$ be a disc, which intersects transversally (exactly)
one of the $E_v$'s at a point $p$, where $p\not\in \widetilde{C}\cap E$.
Set $(D^+,o):=\pi(\widetilde{D}^+,p)\subset (X,o)\subset (\CC^{r+1},o)$.
Then we claim that $(C\cup D^+,o)$ is an ordinary $(r+1)$--tuple in $(\CC^{r+1},o)$.
(Indeed, as above, since $(Z_{\min}, \widetilde{D}^+)=1$, the intersection multiplicity of $\ell$ with $D^+$ is one.
This shows that $D^+$ is smooth  and also that $(C,D^+)_{Hir}=1$. Then, by Hironaka's formula~\eqref{eq:deltaA1}
one obtains $\delta(C\cup D^+)=\delta(C)+(C,D^+)_{Hir}=r$, and the result follows by~\ref{ex:delta22}.)

\subsubsection{} \label{bek:rat3}
Note that, if we start with an arbitrary collection
$\widetilde{C}_{Z_{\min}}=\cup_{i=1}^r\widetilde{C}_i$ of smooth discs of ${\tilde X}$, transversal to $E$, such that
$(\widetilde{C}_{Z_{\min}}, E_v)=-(Z_{\min},E_v)$ for every $v$, then (since ${\rm Pic^0}({\tilde X})=0$),
there exists a function  $\ell $ as above such that $\pi(\widetilde {C}_{Z_{\min}})=\{\ell=0\}\cap(X,o)$,
hence $\pi(\widetilde{C}_{Z_{\min}})$ is an
ordinary $r$--tuple. Moreover, one can add an arbitrary transversal $\widetilde{D}^+$ anywhere generically along $E$,
then  $\pi(\widetilde{C}_{Z_{\min}}\cup \widetilde{D}^+) $ is an ordinary $(r+1)$--tuple.
Clearly, any subset of $k$ components of  $\pi(\widetilde{C}_{Z_{\min}}\cup \widetilde{D}^+) $ is an ordinary $k$--tuple.
\subsubsection{}\label{bek:rat4}
It turns  out that besides curves of type $\pi(\widetilde{C}_{Z_{\min}}\cup \widetilde{D}^+) $ constructed as above in \ref{bek:rat2}
and \ref{bek:rat3}, there
are other collections (not subsets of $\pi(\widetilde{C}_{Z_{\min}}\cup \widetilde{D}^+) $ type curves) which are still
ordinary $k$--tuples, see subsection \ref{eq:sh}. Our next goal is to list all the possibilities in the case of cyclic quotient singularities, using a more subtle method (versus the above
generic hyperplane section property).

Before we start the next construction let us stress again the following fact (compare also with \ref{ss:first}).
In the argument used above (when we used
the pair $\ell$ and $\ell'$), it was helpful to consider several  linear functions. In order to `test' a disc in
${\widetilde X}$, transversal to $E_v$, we need functions, whose divisor have $E_v$--multiplicity one.
In particular, what we need to understand is the set of family of functions on $(X,o)$
whose divisors on ${\widetilde X}$ have multiplicity one along at least one component of $E$.

Though the  next construction can be done more generally, here we present it for cyclic
quotient  singularities, when we can compare it with the classical theory of such singularities.

\section{Cyclic quotient singularities}

\subsection{Notation}\label{ss:NOT}\

Let us fix a cyclic quotient singularity $(X,o)$. It can be characterized by several different facts,
here we will use simultaneously a few of them (for details see e.g. \cite{Hirzebruch,BPV-book} or~\cite[2.3]{NBook}).

\begin{enumerate}[label=(\alph*)]
 \item\label{cyclic:1}
 $(X,o)$ is isomorphic with the normalization of $(\{xy^{n-q}=z^n\},0)\subset (\CC^3,0)$, where $0<q<n$, $(n,q)=1$;
 \item\label{cyclic:2}
 $(X,o)$ is the quotient singularity $(\mathbb{C}^2,o)/\mathbb{Z}_n$ of $\mathbb{Z}_n=\{\xi\in\mathbb{C}\,:\, \xi^n=1\}$,
where the action is $\xi*(z_1,z_2)=(\xi z_1,\xi^q z_2)$ (where $(n,q)$ is as in~\ref{cyclic:1});
 \item\label{cyclic:3}
 $(X,o)$ is the germ at $o$ of ${\rm Spec}\, \mathbb{C}[\sigma^\vee\cap \mathbb{Z}^2]$, where $\sigma^\vee$ is the real cone
generated in $\mathbb{R}^2$ by the vectors $(1,0)$ and $(q,n)$ (where $(n,q)$ is as in~\ref{cyclic:1});
 \item\label{cyclic:4}
 $(X,o)$ is a surface singularity whose minimal resolution dual graph is a string (or bamboo) whose vertices have all genus zero.
\end{enumerate}

\vspace{2mm}

Regarding part~\ref{cyclic:2}, the invariant ring $\mathbb{C}[z_1,z_2]^{\mathbb{Z}_n}$ is generated by several
elements which will be discussed below, here we mention only three of them: $x=z_2^n$, $y=z_1^n$ and $z=z_1^{n-q}z_2$,
they satisfy $xy^{n-q}=z^n$ (the equation in~\ref{cyclic:1}).
Below we will exploit several further connections between these equivalent characterizations.

Regarding part~\ref{cyclic:4}, we denote the irreducible exceptional divisors of the minimal resolution by
$E_1, \ldots, E_s$. The (reduced) strict transform of $x=0$ (lifted to the normalization) is a transversal cut of $E_1$,
it will be denoted by $\widetilde{S}$, and the (reduced) strict transform of $y=0$ (lifted to the normalization) is a
transversal cut of $E_s$, it will be denoted by $\widetilde{S}'$. The self-intersecion numbers $-k_v\leq -2$ of the
irreducible exceptional divisors $E_v$ are given by the Hirzebruch-Jung continued fraction of $n/q=[k_1,\ldots, k_s]$.
Using the graph and the self-intersections we have the intersection form $(\,,\,)$ on the lattice
$L=\mathbb{Z}\langle E_v\rangle_v$. We write
$L':=\{l'\in L\otimes\mathbb{Q}\,:\, (l',l)\in \mathbb{Z} \ \mbox{for any $l\in L$}\}$
for the lattice of `dual' rational cycles supported on $E$. It is generated by the dual base elements
$\{E^*_v\}_v$ defined by $(E^*_v,E_w)=-\delta _{v,w}$ (the negative of Kronecker delta) for all $v$ and $w$. Their geometric meaning
is the following:
If $\widetilde{C}_v $ is a generic transversal cut of $E_v$, then (the rational divisor) $\widetilde{C}_v+E^*_w $
intersects any element of $L$ trivially. We also note that $L\subset L'$ and $L'/L$ can be identified with the first
homology $H$ of the link of $(X,o)$, which is $\mathbb{Z}_n$.

\subsection{A family of ordinary $r$--tuples embedded in a cyclic quotient singularity}\label{eq:sh} \

Since cyclic quotient singularities are minimal rational, the construction from~\ref{s:embcurves} holds for them,
hence in that way we know that $\pi(\widetilde{C}_{Z_{\min}}\cup \widetilde{D}^+) $ is an $(r+1)$--tuple in $(X,o)$,
where $\widetilde{C}_{Z_{\min}}+E$ intersects any $E_v$ trivially, and $r=-Z_{\min}^2=-E^2$.

Next, we recall another family of reduced  Weil divisors, which form ordinary $r$--tuples,
realized by a completely different construction.

First, we recall some general facts, valid for any resolution (when the link is a rational homology sphere, that is,
the graph $\Gamma$ is a tree and all the genera are zero). For details the reader is referred to~\cite{NBook}.
Let us consider  the rational Lipman cone $\mathcal {S}':=\{l'\in L'\,:\, (l',E_v)\leq 0\ \mbox{for all $v$}\}$.
It is known that if $s'\in\mathcal{S}'\setminus \{0\}$ then all the $E_v$-coefficients of $s'$ are strictly positive.
Moreover, $\mathcal {S}'$ is generated over $\mathbb{Z}_{\geq 0}$ by the cycles $\{E_v^*\}_v$. Let us denote by $[l']$
the class of $l'\in L'$ in $H=L'/L$. Then, for any $h\in H$ there exists a unique minimal element
$s_h\in \mathcal{S}'$ such that $[s_h]=h$. Here minimality is considered with respect to the partial ordering:
$l'\leq l''$ ($l',l''\in L'$) if and only if all the $E_v$--coefficients satisfy $l'_v\leq l''_v$.

For example, $s_0=0$, but $s_h\not=0$  for  any $h\not=0$. In particular, for $h\not=0$, any $s_h$ can be written as
$s_h=\sum_v a_v E^*_v$ for some $a_v\in\mathbb{Z}_{\geq 0}$, not all zero. Associated with this description, we consider
$\widetilde{C_h}$ in $\widetilde{X}$: it has $r:=\sum_va_v$ irreducible components, $a_v$ of them constitute generic
transversal cuts of $E_v$. (Hence $\widetilde{C_h}$ depends on the choice of these cuts, but what we will claim for
$\widetilde{C_h}$ will be independent of this choice).

Let $(C_h,o)$ be $\pi(\widetilde{C_h})$ in $(X,o)$ associated with $h\in H\setminus \{0\}$ (we call such a curve $(C_h,o)$
a \emph{minimal generic $h$-curve}).
Minimal generic $h$-curves were introduced in~\cite{CLMN-santalo} in the context of surface singularities with
a $\QQ HS^3$ link.

In \cite[Thm. 5.1]{CLMN-santalo} (see also \cite[Thm. 0.1, 0.2]{CM}) it is shown that $\delta(C_h,o)=r-1$,
hence $(C_h,o)$ is an $r$--tuple.

In fact, this result is true not only for cyclic singularities but for a more general class including the especial
singularities $\EE_6$, $\EE_7$ (see \cite[Thm. 5.1]{CLMN-santalo}).

\subsubsection{}

For an arbitrary graph it is not simple to determine the cycles $s_h$. For cyclic quotients (and star shaped graphs) it is determined in
\cite{NemOSZ}, see also \cite{NBook}. Here, for the cyclic quotient case,  we follow the
 algorithm  from  \cite[10.3.3]{NemOSZ}, which provides the coefficients $\{a_v\}_{v=1}^s$.

 In the cyclic quotient case, $H=\mathbb{Z}_n$, it is generated by $[E_s^*]$, hence any $h$ has the form
 $a[E^*_s]$ for some $0\leq a<n$. Accordingly,  for
   any $a$, $0\leq a<n$,
   we search for the elements $s_{a[E^*_s]}$ and their coefficients $\{a_v\}_v$ (they depend on $h$ as well).

By  \cite[10.3.3]{NemOSZ},  all the possible   entries $(a_1,\ldots , a_s)$, depending on $h$,
  can be  generated inductively as follows. (For a different formula see \cite{CM}.)

We start with the entries $(k_1-1, k_2-2, \ldots, k_s-2)$. This $s$--tuple  represents  $s_{(n-1)[E^*_s]}$.
This means that $[(k_1-1)E^*_1+\sum_{v\geq 2} (k_v-2)E^*_v]=(n-1)[E^*_s]$ in $H=\mathbb{Z}_n$, and it is the minimal lift of
$(n-1)[E^*_s]$ into $\mathcal{S}'$.

Then  we determine inductively the $\{a_v\}_v$ system of $a-1$ from the entries of $a$.

If $(a_1, \ldots, a_s)$ represents $s_{a[E^*_s]}$,
 then the entries of $s_{(a-1)[E^*_s]}$ are determined as follows. If $a_s>0$ then the $s$--tuple for $a-1$ is
$(a_1, \dots, a_{s-1}, a_s-1)$.  If $a_u=a_{u+1}=\cdots =a_s=0$, but $a_{u-1}\not=0$,  then the $s$--tuple of $a-1$ is
$(a_1,\ldots , a_{u-1}-1, k_u-1, k_{u+1}-2, \ldots, k_s-2)$.
The last term is $(0,\ldots, 0)$ (corresponding to $a=0$) when we stop.

\begin{example}\label{ex:1511:1}
Assume that $n/q=15/11=[2,2,3,2,2]$. Then this algorithm produces (in the next order) the $5$--tuples
$$(1,0,1,0,0), \  (1,0,0,1,0), \ (1,0,0,0,1), \ (1,0,0,0,0), \ (0,1,1,0,0), $$
$$(0,1,0,1,0), \ (0,1, 0,0,1), \  (0,1,0,0,0), \  (0,0,2,0,0), \ (0,0,1,1,0), $$
$$(0,0, 1,0,1), \ (0,0,1,0,0),\ (0,0,0,1,0), \ (0,0,0,0,1), \ (0,0,0,0,0). $$

Note that this family in general  is not the same as  $\pi(\widetilde{C}_{Z_{\min}}\cup \widetilde{D}^+) $ constructed in
   \ref{s:embcurves}. In the above lists all the curves have one or two components, while
   $\pi(\widetilde{C}_{Z_{\min}}\cup \widetilde{D}^+) $
   has four. On the other hand, $(0,1,0,1,0)$ is in the above list, but it cannot be represented as a subset of
     $\pi(\widetilde{C}_{Z_{\min}}\cup \widetilde{D}^+) $. (In this case $Z_{\min}=E=E_1^*+E^*_3+E^*_5$, hence in the above language it is $(1,0,1,0,1)$,
     and  $\pi(\widetilde{C}_{Z_{\min}}\cup \widetilde{D}^+) $ is  $(1,0,1,0,1)$ completed by an additional 1 at any position.)
\end{example}

The goal of the next parts --- the main results of the note --- is to provide a common generalization
of these two families; in fact,
{\it we give  a complete description of the $r$--tuple Weil divisors in cyclic quotient singularities. }

\section{Cyclic quotient singularities. The new results.}
\subsection{The `canonical'  embedding of the cyclic quotient singularities}\label{ss:class}\

In this subsection we recall some well--known facts regarding cyclic quotient singularities, what we will need in the main construction,
see e.g. \cite[2.3]{NBook}.

Consider the cyclic quotient singularity $(X,o)=(\CC^2,o)/\ZZ_n$ with the action
$\xi*(z_1,z_2)=(\xi z_1,\xi^qz_2)$, where $0<q<n$ and $(q,n)=1$.
Then $\cO_{(X,o)}=\CC[z_1,z_2]^{\ZZ_n}$, is the ring of invariants.
The invariant monomials $z_1^Nz_2^M$, determined by $N+qM\equiv 0 \ ({\rm mod} \ n)$ play a key role.
Furthermore, the minimal set of such invariant monomials `canonically' generate
the invariant ring, producing a `canonical' way to embed $(X,o)\subset (\CC^{r+1},o)$.

Let us first describe how we represent an invariant monomial in the minimal resolution.
The following diagram shows the divisor of $\pi^*(z_1^Nz_2^M)$ on $\tX$, where $m_v$ are the vanishing orders
along the irreducible exceptional curves $E_v$,  the left (resp. the right) arrow represents the reduced strict transform
$\widetilde{S}$ (resp. $\widetilde{S}'$)
of $\{x=z_2^n=0\}$ (resp. of $\{y=z_1^n=0\}$), and the multiplicities $m_0$ and $m_{s+1}$ represent
the vanishing orders of $\pi^*(z_1^Nz_2^M)$ along $\widetilde{S}$  and $\widetilde{S}'$, that is,
 $m_0=M$ and $m_{s+1}=N$.

\begin{picture}(300,40)(10,10)
\put(112,25){\makebox(0,0)[r]{$(m_0)$}}
\put(300,25){\makebox(0,0)[l]{$(m_{s+1})$}}
\put(150,35){\makebox(0,0){$-k_1$}}
\put(190,35){\makebox(0,0){$-k_2$}}
\put(260,35){\makebox(0,0){$-k_s$}}
 \put(150,15){\makebox(0,0){$(m_1)$}}
\put(190,15){\makebox(0,0){$(m_2)$}}
\put(260,15){\makebox(0,0){$(m_s)$}}
 \put(150,25){\circle*{4}}
\put(190,25){\circle*{4}} \put(260,25){\circle*{4}}
\put(210,25){\vector(-1,0){90}} \put(240,25){\vector(1,0){50}}
\put(225,25){\makebox(0,0){$\cdots$}}
\end{picture}

\vspace{3mm}

\noindent
The integers  $\{m_v\}_{v=0}^{s+1}$ satisfy
\begin{equation}\label{eq:conditions}\left\{\begin{array}{l}
m_v\geq 1 \ (1\leq v\leq s),  \  m_0\geq 0,  \  \textrm{and} \   m_{s+1}\geq 0,\\
m_{v-1}-k_vm_v+m_{v+1}=0 \ \ \mbox{for all $1\leq v\leq s$}.
\end{array}\right.
\end{equation}
\noindent
We call such a collection of integers  $\{m_v\}_{v=0}^{s+1}$ satisfying (\ref{eq:conditions}) a \emph{multiplicity system}.
Given a multiplicity system $\{m_v\}_{v=0}^{s+1}$, using the conditions in~\eqref{eq:conditions}, the
entries $(m_0,m_{s+1})$ (as well as $(m_0,m_1)$) uniquely determine~$\{m_v\}_{v=0}^{s+1}$.

Let us recall next how one finds the minimal set of  generators of $\CC[z_1,z_2]^{\ZZ_n}$.

Recall that   $\sigma^{\vee}$ denotes the cone ${\RR_{\geq 0}}\langle (1,0), (q,n)\rangle$.
Then
$(X,o)$ is realized as the toric singularity $\Spec \CC[\sigma^{\vee}\cap \ZZ^2]$, that is
$\CC[z_1,z_2]^{\ZZ_n}= \CC[\sigma^{\vee}\cap \ZZ^2]$. Furthermore, consider the convex hull of the
lattice points of $\sigma^{\vee}\cap \ZZ^2\setminus \{0\}$. Its boundary will consist of bounded edges (a finite number of
edges connecting $(1,0)$ and $(q,n)$) and two unbounded edges (the rays  $\{(1,0)\}+\RR_{\geq 0}\langle (1,0)\rangle$
and $\{(q,n)\}+\RR_{\geq 0}\langle (q,n)\rangle$). Consider the sequence of lattice points
$\{\mathfrak{p}_j\}_{j=0}^{t+1}:=\{(p_1^j, p_2^j)\}_{j=0}^{t+1}$ on the bounded edges, say $(p_1^0,p_2^0)=(1,0)$, $(p_1^1,p_2^1)=(1,1)$, and
$(p_1^{t+1}, p_2^{t+1})=(q,n)$. Interpret the vectors $(p_1^j,p_2^j)$ as $(m^j_1, m^j_0)$.
Then, it turns out that  each pair $(m^j_1, m^j_0) $ can be completed uniquely to a multiplicity system $\{m_v^j\}_{v=0}^{s+1}$.

A short way to see this is the following.
The entries $(m^j_0,m^j_{s+1})$ can be recovered as the scalar products $\langle (p^j_1,p^j_2), (0,1)\rangle $ and
 $\langle (p^j_1,p^j_2), (n,-q)\rangle$  respectively, both are non-negative, at least one of them is positive,   and
  the monomial $g_j:= z_2^{m^j_0} z_1^{m^j_{s+1}}$ is invariant,
hence it belongs to the ring $\cO_{(X,o)}$ of $(X,o)$, and its divisor is given by the multiplicity system
$\{m^j_v\}_{v=0}^{s+1}$.

In fact, even more is true: all the entries of $\{m^j_v\}_{v=0}^{s+1}$ can be recovered from the toric combinatorics as  follows. Let $\sigma
=\RR_{\geq 0}\langle (n,-q), (0,1)\rangle$ be the dual cone of $\sigma^\vee$ (see also Remark \ref{remarkab}).
Then similarly as we constructed the lattice points $\{\mathfrak{p}_j\}_{j=0}^{t+1}$ in the cone $\sigma ^\vee$, we can construct the
points $\{\mathfrak{e}_v\}_{v=0}^{s+1}$ in $\sigma$. Then $\mathfrak{e}_0=(0,1)$ corresponds to $\widetilde{S}$,
 $\mathfrak{e}_1=(1,0)$ corresponds to $E_1$,  $\mathfrak{e}_v$ corresponds to $E_v$ ($1\leq v\leq s$), and
  $\mathfrak{e}_{s+1}=(n,-q)$ corresponds to $\widetilde{S}'$. Then the multiplicity $m_{v}^j$ equals
  the scalar product $\langle \mathfrak{p}_j,\mathfrak{e}_v\rangle$.
  (That is, the lattice points  $\{\mathfrak{p}_j\}_{j=0}^{t+1}$ correspond to the  invariant monomial  generators $\{g_j\}_j$,  while
   the lattice points
$\{\mathfrak{e}_v\}_{v=0}^{s+1}$ to the irreducible divisors $\widetilde{S}, E_1,\ldots , E_s, \widetilde{S}'$.)

For instance, the invariant monomials $g_0=z_1^n$, $g_1=z_1^{n-q}z_2$, and
$g_{t+1}=z_2^n$ correspond to the lattice points $(1,0), \ (1,1), \ (q,n)$ and are represented as follows:

\begin{picture}(300,40)(-30,10)
\put(72,25){\makebox(0,0)[r]{$\pi^*(y:=g_0=z_1^n)$}}
\put(112,25){\makebox(0,0)[r]{$(0)$}}
\put(300,25){\makebox(0,0)[l]{$(n)$}}
\put(150,15){\makebox(0,0){$(1)$}}
\put(260,15){\makebox(0,0){$(q')$}}
\put(150,25){\circle*{4}}
\put(190,25){\circle*{4}} \put(260,25){\circle*{4}}
\put(210,25){\vector(-1,0){90}}
\put(240,25){\vector(1,0){50}}
\put(225,25){\makebox(0,0){$\cdots$}}
\end{picture}

\begin{picture}(300,40)(-30,10)
\put(72,25){\makebox(0,0)[r]{$\pi^*(z:=g_1=z_1^{n-q}z_2)$}}
\put(112,25){\makebox(0,0)[r]{$(1)$}}
\put(300,25){\makebox(0,0)[l]{$(n-q)$}}
\put(150,15){\makebox(0,0){$(1)$}}
\put(150,25){\circle*{4}}
\put(190,25){\circle*{4}} \put(260,25){\circle*{4}}
\put(210,25){\vector(-1,0){90}} \put(240,25){\vector(1,0){50}}
\put(225,25){\makebox(0,0){$\cdots$}}
\end{picture}

\begin{picture}(300,40)(-30,10)
\put(72,25){\makebox(0,0)[r]{$\pi^*(x:=g_{t+1}=z_2^n)$}}
\put(112,25){\makebox(0,0)[r]{$(n)$}}
\put(300,25){\makebox(0,0)[l]{$(0)$}}
\put(150,15){\makebox(0,0){$(q )$}}
\put(260,15){\makebox(0,0){$(1)$}}
\put(150,25){\circle*{4}}
\put(190,25){\circle*{4}} \put(260,25){\circle*{4}}
\put(210,25){\vector(-1,0){90}} \put(240,25){\vector(1,0){50}}
\put(225,25){\makebox(0,0){$\cdots$}}
\end{picture}

\vspace{2mm}
\noindent
where $0<q'<n$ satisfies $qq'\equiv 1\ ({\rm mod}\ n)$.

These functions satisfy the identity $xy^{n-q}=z^n$, hence there is a map
$\tX\to Y:=\{xy^{n-q}=z^n\}\subset \CC^3$. Note that $(Y,o)$ is not normal, its normalization
is exactly $(X,o)$, in the normalization we have to add (at least) all the invariant monomials~$\{g_j\}_{j=0}^{t+1}$.

One of the claim of the theory of cyclic quotients  is that $\{g_j\}_{j=0}^{t+1}$ constitute a minimal set of generators of the invariant ring,
hence they embed $(X,o)$ into $\CC^{t+2}$. They are the coordinates of $\CC^{t+2}$. Their number is
$t+2=-Z_{\min}^2+1=r+1$. Since $Z_{\min}=\min_{j} \div_E(g_j)=E$, at every vertex $v$
there should exist a generator $g_j$ whose multiplicity $m^j_v=1$. In fact,
for any $g_j$, at least one of the terms of its multiplicity sequence should be 1.
This can be seen as follows. By~\cite{Artin66}, for any function $f\in \cO_{(X,o)}$ one has
that $f\in\m_{(X,o)}^2$ if and only if $\div_E(\pi^*f)\geq 2E$. Hence, if for some fixed $g_j$ one had
$m^j_v\geq 2$ for all $v=1,\ldots,s$, then $g_j\in \m_{(X,o)}^2$ and hence it could be eliminated from
the functions defining the minimal embedding. But, as we already said, all functions $\{g_j\}_{j=0}^{t+1}$
are needed for the minimal embedding since $t+2=-Z_{\min}^2+1={\rm embdim}(X,o)$.

\subsection{A new way to generate the functions $\{g_j\}_j$}\label{ss:new}\

In the next discussion we will present a rather  different algorithm, which provides
the coordinates $\{g_j\}_j$ of $(\mathbb{C}^{t+2},o)$ via their multiplicity system (versus the lattice point search in $\sigma^{\vee}$
described in the previous subsection).
The starting point aims  to find all multiplicity systems, which have at least one entry equal to 1. (We already know that the multiplicity systems of the invariant monomials have this property, here we ask if they are all of them. It turns out that indeed, there are no other systems. This is a very special property of cyclic quotients, for more general singularities this is not the case.)
The new picture will have several additional important outputs as well.

First, for any fixed $E_u$, let us recall the following ideal filtration of $\cO_{(X,o)}$:
$$
\cF(mE_u)=\{f\in\cO_{(X,o)} \mid \div_E(\pi^*(f))\geq mE_u\}.
$$
This filtration appeared in a more general setting in~\cite[\S~2.7.]{kappa2}.
Note that $\cF((m+1)E_u)\subset\cF(mE_u)$ and $\cF(mE_u)/\cF((m+1)E_u)$ is a finite dimensional vector space.
Note also that $\cF(E_u)=\{f\in\cO_{(X,o)} \mid \div_E(\pi^*(f))\geq E\}=\m_{(X,o)}$.

\begin{prop}
\label{prop:dim}
For any fixed vertex $u$ of the minimal resolution graph of  a rational surface singularity $(X,o)$
(with Euler number $-k_u$ and valency $\val_u$)
$$
\dim_\CC\cF(E_u)/\cF(2E_u)=k_u-\val_u+1.
$$
\end{prop}

\begin{proof}
Note that
$$
\cF(E_u)/\cF(2E_u)=\frac{H^0(\tX,\cO_{\tX}(-E))}{H^0(\tX, \cO_{\tX}(-E-E_u))}.
$$
Our purpose is to calculate the dimension of the right hand side.
For this, consider the exact sequence
$$0\to \cO_{\tX}(-E_u-E)\to  \cO_{\tX}(-E)\to  \cO_{E_u}(-E)\to 0.$$
We claim that $h^1(\cO_{\tX}(-E))=
h^1(\cO_{\tX}(-E_u-E))=0$. The first part follows from the exact sequence
$0\to \cO_{\tX}(-E)\to  \cO_{\tX}\to  \cO_{E}\to 0$ and from  $h^1(\cO_{\tX})=p_g(X,0)=0$, $H^1(\cO_E)\simeq\CC$.
For the second part note that by Laufer's algorithm there exists a computation sequence $\{x_i\}_i$ connecting
$E+E_u$ with $s(E+E_u)$ (cf.~\cite[\S~2.2.1]{kappa2}) such that $(x_i, E_{v(i)})=1$ at every step.
Then, $h^1$ is stable along the sequence, and $h^1(\cO_{\tX}(-s))=0$ for any $s\in \cS$ by
the Lipman Vanishing Theorem~\cite{Lipman}. Hence
$$\dim \frac{H^0(\tX,\cO_{\tX}(-E))}{H^0(\tX, \cO_{\tX}(-E-E_u))}=\chi(\cO_{E_u}(-E))=-(E_u,E)+1=k_u-\val_u+1.$$
\end{proof}

From this it is not clear a priori that the class of invariant monomials might generate (freely) this space.
But the following construction will show exactly this fact.

As a first step we find all the  multiplicity systems $\{m_v\}_{v=0}^{s+1}$ with $m_u=1$. (Next it is convenient to assume that $s\geq 2$, otherwise the indexing should be adapted.)

If $m_u=1$ then   $m_{u-1}+m_{u+1}=k_u$.
If $u=1$, then we can take for $m_{u-1}=m_0$ the values $0, 1, \ldots , k_u-1$ (since $m_{u+1}\geq 1$). If $u$ is not an end, then
$m_{u-1}$ can be  $1, \ldots , k_u-1$. Any such choice can be complete in a unique  way to a multiplicity system by the following lemma
(applied for $u=w$ with $m_u=1$).

\begin{lemma}
\label{lemma:auxmv}  Let us fix a vertex $w$, $1\leq w\leq s$.

\begin{enumerate}[label=(\alph*)]
 \item\label{mult_sequence:1}
 Assume that a pair $(m_w,m_{w+1})\in \mathbb{Z}_{>0}^2$ is fixed  with $m_w\leq m_{w+1}$. Then we can find non-negative
multiplicities $\{m_v\}_{v=w+2}^{s+1}$ such that $m_{v-1}-k_vm_v+m_{v+1}=0$ for any $s\geq v>w$.
 \item\label{mult_sequence:2}
 Symmetrically, assume that the pair $(m_{w-1},m_{w})\in \mathbb{Z}_{>0}^2$
is fixed  with $m_{w-1}\geq m_{w}$. Then we can find non-negative
multiplicities $\{m_v\}_{v=0}^{w-2}$ such that $m_{v-1}-k_vm_v+m_{v+1}=0$ for any $1\leq v<w$.
\end{enumerate}
\end{lemma}

\begin{proof}
To check \ref{mult_sequence:1} note that $m_{w+2}=k_{w+1}m_{w+1}-m_w\geq (k_{w+1}-1)m_{w+1}\geq m_{w+1}$,
the latter inequality follows since $k_v\geq 2$. Hence $m_{w+2}$ is non-negative and shows \ref{mult_sequence:1}
recursively. Part~\ref{mult_sequence:2} is similar.
\end{proof}

Then, the multiplicity system just found identifies a divisor supported on $\widetilde{S}\cup E\cup \widetilde{S}'$,
hence an invariant monomial. In particular, it is generated by the monomials $\{g_j\}_j$, cf. subsection \ref{ss:class}.
However, since  $m_u=1$, it should be one of the monomials $g_j$'s (up to a nonzero constant). Note that the number of
monomials realized in this way  is exactly $k_u-\val_u+1$. Moreover,  these  monomial functions  are linearly independent
modulo $\m_{(X,0)}^2$ (since all the $g_j$ monomials  are).
Hence, monomial functions constructed in this way (i.e., the $g_j$'s with $m^j_u=1$) generate freely
$H^0(\tX,  \cO_{\tX}(-E))/H^0(\tX, \cO_{\tX}(-E-E_u))$.

Now, if we run the above construction for any $1\leq u\leq s$, and we wish to reconstruct all the possible
multiplicity systems with at least one term $m_u=1$, then, in fact, we recover exactly the multiplicity system
of all monomial functions $\{g_j\}_j$.

However, if we consider all the possible choices of vertices $u$, then we get some coincidences:
the choice $(m_{u-1},m_{u+1})=(k_u-1,1)$ for $m_u=1$ generates the same system as $(m_{u},m_{u+2})=(1,k_{u+1}-1)$
for $m_{u+1}=1$. But these type of coincidences generate all the possible coincidences.

Thus, in this way we can generate all the monomials $g_j$ and even we can put them in order, ordered increasingly
by $m_0$. (This order coincides with the order of the lattice points on the boundary path--segment considered above.)

\begin{example}\label{ex:22322}
Consider the graph of the cyclic singularity $(15,11)$ from Example~\ref{ex:1511:1}.
Note that $\mathbf k=[2,2,3,2,2]$ is the sequence of (opposites of) self-intersecions of the exceptional divisors
of the dual resolution graph. We have the following table of possible multiplicity sequences.
In the horizontal rows instead of the complete graph we put only the two arrows, and the multiplicities on the place
of the vertices of the graph. On the left column we indicate the invariant monomial as well.
With boldface we indicate the zigzag of multiplicity 1 entries.
The first two columns indicate the lattice points $(1,0), \ (1,1), \ (3,4),\ (11,15)$
which generate the semigroup~$\sigma^\vee\cap \ZZ^2$.

\begin{picture}(300,80)(40,-40)
\put(72,25){\makebox(0,0)[r]{$z_1^{15}$}}
\put(112,25){\makebox(0,0){$0$}}
\put(150,25){\makebox(0,0){${\bf 1}$}}
\put(190,25){\makebox(0,0){$2$}}
\put(230,25){\makebox(0,0){$3$}}
\put(270,25){\makebox(0,0){$7$}}
\put(310,25){\makebox(0,0){$11$}}
\put(140,25){\vector(-1,0){20}} \put(320,25){\vector(1,0){20}}
\put(348,25){\makebox(0,0){$15$}}

\put(72,10){\makebox(0,0)[r]{$z_1^{4}z_2$}}
\put(112,10){\makebox(0,0){$1$}}
\put(150,10){\makebox(0,0){${\bf 1}$}}
\put(190,10){\makebox(0,0){${\bf 1}$}}
\put(230,10){\makebox(0,0){${\bf 1}$}}
\put(270,10){\makebox(0,0){$2$}}
\put(310,10){\makebox(0,0){$3$}}
\put(140,10){\vector(-1,0){20}} \put(320,10){\vector(1,0){20}}
\put(348,10){\makebox(0,0){$4$}}

\put(72,-5){\makebox(0,0)[r]{$z_1z_2^4$}}
\put(112,-5){\makebox(0,0){$4$}}
\put(150,-5){\makebox(0,0){$3$}}
\put(190,-5){\makebox(0,0){$2$}}
\put(230,-5){\makebox(0,0){${\bf 1}$}}
\put(270,-5){\makebox(0,0){${\bf 1}$}}
\put(310,-5){\makebox(0,0){${\bf 1}$}}
\put(140,-5){\vector(-1,0){20}} \put(320,-5){\vector(1,0){20}}
\put(348,-5){\makebox(0,0){$1$}}

\put(72,-20){\makebox(0,0)[r]{$z_2^{15}$}}
\put(112,-20){\makebox(0,0){$15$}}
\put(150,-20){\makebox(0,0){$11$}}
\put(190,-20){\makebox(0,0){$7$}}
\put(230,-20){\makebox(0,0){$3$}}
\put(270,-20){\makebox(0,0){$2$}}
\put(310,-20){\makebox(0,0){${\bf 1}$}}
\put(140,-20){\vector(-1,0){20}} \put(320,-20){\vector(1,0){20}}
\put(348,-20){\makebox(0,0){$0$}}

\end{picture}
\end{example}

\begin{remark}\label{remarkab}  (a)
The zigzag of multiplicity 1 entries has the following shape: the first column has $k_1$ places, the second one
$k_2-1$ places, however the last place of the first column is in the same row as the first place of the second one.
And similar rule applies for the other columns as well. In particular, all together there are
$\{k_1+(k_2-1)+\ldots + (k_{s-1}-1)+k_s\}-(s-1)=1+\{  (k_1-1)+(k_2-2)+\ldots + (k_{s-1}-2)+(k_s-1)\}$
rows, whose number is $t+2=r+1=-Z_{\min}^2+1$ (i.e. the number of $g_j$'s).

(b) Let us consider the dual cone $\sigma $ of $\sigma ^\vee$,
 $\sigma=\RR_{\geq 0}\langle (n,-q), (0,1)\rangle$, which determines (by a similar way) the dual graph as well.
 Then,
from this diagram one can read the self--intersections of the dual graphs, or the entries
$[n/(n-q)]=[d_1,\ldots, d_t]$. Indeed for any $j=1, \ldots, t$, one has the dual identities
$m^{j-1}_v-d_jm^j_v+m^{j+1}_v=0$ (for any $v$). In fact for $j=1,\ldots , t$ the integer
 $d_j-1$ is exactly  the number of
1's of the zigzag sitting in the row $j$. In the above Example \ref{ex:22322}
the dual sequence is $15/4=[4,4]=[d_1,d_2]$.
In this way the zigzag of 1's can be identified with Riemenschneider's point diagram (providing the above duality)
\cite{R74}.
\end{remark}

\subsection{An analytic plumbing realization of $\tX$ and the functions $g_j$ on it}\label{ss:Pl} \

\subsubsection{}
We would like to emphasize an important point regarding the results of the previous
subsection (and the coming one).
In the previous section we constructed a very precise base for any
$\cF(-E_u)/\cF(-2E_u)$ and also for $\m_{(X,o)}/\m_{(X,o)}^2$. These monomial functions have the following crucial property: we will be able to cover the resolution space $\tX$ with local coordinate systems such that around any point of $\tX$ we will be able to represent in one of the local coordinates {\it simultaneously}  all the monomial functions. (That is, in our application, regarding a base of $\m_{(X,o)}/\m_{(X,o)}^2$,
we will need definitely more than just the multiplicity systems of the base elements. We will need that all the strict transforms are supported on the very same $\widetilde{S}\cup\widetilde{S}'$, and we will need
simultaneous equations of all of them in the very same local coordinate system around any point.)
In the previous subsection we identified the monomial functions as a base of  $\m_{(X,o)}/\m_{(X,o)}^2$ (with some additional properties regarding their multiplicities). In this subsection we will construct certain  local coordinates in which we can write simultaneously all the
equations of the monomial functions.
\subsection{}
In the sequel  we construct the resolution space $\tX$ in such a way that we will have canonical  local charts in which we can represent
  all the coordinate functions $g_j$ simultaneously.

Recall from the previous parts  that $s$ is the number of vertices in the
dual resolution graph of $(X,o)$.
Consider
$s+1$ copies  $U_0,\ldots, U_s$ of $\CC^2$  with affine coordinates $(\alpha_v,\beta_v)$, $0\leq v\leq s$.
For each $1\leq v\leq s$ we glue $U_{v-1}\setminus \{\alpha_{v-1}=0\}$ with
$U_v\setminus \{\beta_v=0\}$ by
$\beta_v=\alpha_{v-1}^{-1}$ and $\alpha _v=\alpha_{v-1}^{k_v}\beta_{v-1}$.
This way, the curve $E_v$ is covered by the union $U_{v-1}\cup U_v$ and it is given by
$\{\beta_{v-1}=0\}$  in $U_{v-1}$  and $\{\alpha_v=0\}$ in $U_v$. It has self--intersection number $-k_v$.
 Hence $\cup_{v=1}^sE_v$ has dual resolution graph $\Gamma$ of the
cyclic quotient singularity we started with. Since the cyclic quotient singularity is taut
(the analytic type is determined by the topology), we can assume that the resolution $\tX$ of
the cyclic quotient singularity $(X,o)$ we started with in the previous sections is exactly the space
$\cup_{v\geq 0}U_v$, and $(X,o)$ is the singularity obtained from $\cup_{v\geq 0}U_v$ by contracting $\cup_{v=1}^sE_v$.

In the sequel we denote this (analytic plumbed)
space by $\tX$ and the resolution/contraction  by $\pi$.

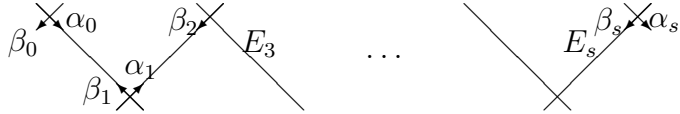
\begin{picture}(300,60)(-10,0)
\put(140,35){\makebox(0,0)[r]{$E_3$}}\put(260,35){\makebox(0,0)[r]{$E_s$}}
\put(264,43){\makebox(0,0){$\beta_s$}}\put(285,43){\makebox(0,0){$\alpha_s$}}
\put(280,50){\vector(-1,-1){10}}\put(270,50){\vector(1,-1){10}}
\put(50,50){\line(1,-1){40}} \put(80,10){\line(1,1){40}}\put(110,50){\line(1,-1){40}}
\put(180,30){\makebox(0,0){$\ldots$}}\put(210,50){\line(1,-1){40}} \put(240,10){\line(1,1){40}}
\put(60,50){\vector(-1,-1){10}}\put(50,50){\vector(1,-1){10}}
\put(45,35){\makebox(0,0){$\beta_0$}}\put(67,43){\makebox(0,0){$\alpha_0$}}
\put(90,10){\vector(-1,1){10}}\put(80,10){\vector(1,1){10}}
\put(73,17){\makebox(0,0){$\beta_1$}}\put(89,25){\makebox(0,0){$\alpha_1$}}
\put(120,50){\vector(-1,-1){10}}
\put(105,43){\makebox(0,0){$\beta_2$}}
\end{picture}

The functions $\pi^*(g_j)$ can be reconstructed as follows. If $\{m_v^j\}_{v=0}^{s+1}$ is the
multiplicity system of $g_j$ then $\pi^*(g_j)$ in the above charts is given by
$$\pi^*(g_j)=\alpha_0^{m^j_0}\beta_0^{m^j_1}=\alpha_1^{m^j_1}\beta_1^{m^j_2} = \cdots = \alpha_v^{m^j_v}\beta_v^{m^j_{v+1}}=\cdots =
\alpha_s^{m^j_s}\beta_s^{m^j_{s+1}}.$$
Eg., the pullback of  $y=g_0$ is
$$\pi^*(g_0)=\alpha_0^0\beta_0=\alpha_1\beta_1^{k_1}=\cdots
.$$
These functions $\Psi:=(g_0,\ldots, g_{t+1}):\tX\to\CC^{t+2}$ embed $(X,o)$ into $\CC^{t+2}$, the image of
$\Psi$ is $(X,o)$.  The semigroup relations between the points $\{(p_1^j,p_2^j)\}_j$ induce the monomial equations of $(X,o)$.
The semigroup relations correspond  to the dual multiplicity relations involving the entries $(d_1, \dots, d_t)$, and the
equations of $(X,o)$  (where now $\{g_j\}_{j=0}^{t+1} $ represent the coordinates of
$\CC^{t+2}$)   are  (cf. \cite{BeRi,StevQ}):
$$g_{j-1}g_{k+1}=g_j^{d_j-1}\cdot (\prod _{j<l<k}g_l^{d_l-2}) \cdot g_k^{ d_k-1} \ \ \ \
(1\leq j\leq k\leq t).$$

\subsection{The embedded curves}\label{ss:embcurve}
Now we arrived at our main point, which motivates all the above discussions.
 Let us consider a disc $\widetilde{D}$ in $\tX$ transversal to $E_v$. It can be given in chart $U_v$ by $\beta_v=c_v$ constant and with
 local coordinate $t_v=\alpha_v$. Then  $\Psi(\widetilde{D})\subset \CC^{t+2}$ is given by the parametrization
 $$t_v\mapsto (g_0(t_v,c_v), \ldots, g_{t+1}(t_v, c_v))=
 (t_v^{m^0_v}c_v^{m^0_{v+1}}, \dots , t_v^{m^{t+1}_v}c_v^{m^{t+1}_{v+1}}).$$
 This necessarily has a linear entry in $t_v$,
 the linear terms are realized by those $g_j$'s, for which  $m^j_v=1$. The constant  parameter  $c_v$ correspond to the position of
 the disc $\widetilde{D}$, it parametrizes the intersection point $\widetilde{D}\cap E_v$. If we forget the non--linear terms, we get the tangent vector
 of $\Psi(\widetilde{D})$ in $\CC^{t+2}$. If $m^j_v=1$ for indices $j_v'\leq j\leq j_v''$ (the position of the 1's in the column $v$),
 then the tangent vector of the parametrization of $\Psi(\widetilde{D})$ is
 $$(0, \ldots, 0,   c_v^{m^{j_v'}_{v+1}}, \ldots ,  c_v^{m^{j_v''}_{v+1}}, 0, \dots, 0). $$
 Note that
 \begin{equation}\label{eq:diff}
 \mbox{ the entries $\{m^j_{v+1}\}_{j_v'\leq j\leq j_v''}$ are all different} \end{equation}
since two neighbors $(m^j_v,m^j_{v+1})$ in a  rows determine the whole multiplicity system of that row (and
  the rows are different).

\subsubsection{}
In the above construction we embedded the disc $\widetilde{D}$ by $\beta_v=c_v\in\mathbb{C}^*$ and
$\alpha_v=t_v$, where $t_v$ is the local parameter. Let us replace this parametrization by  a  general one, but still transversal to
$E_v=\{\alpha_v=0\}$. This has the from
$\beta _v=c_v+c_v't_v+c_v''t_v^2+\cdots$, $\alpha_v=b'_vt_v+b_v''t_v^2+\cdots $ ($b_v'\not=0$).
We let the reader to verify that the tangent line of the newly defined $D_{\gen}=\pi(\widetilde{D}_{\gen})$ is  the same as the tangent line of the
previous $D$. Hence, indeed, our parametrizations covers the case of general transversal cuts completely, even if apparently we use a special
type of parametrization of the cuts.

\subsubsection{}
 Let us place  two transversal discs, corresponding to the positions $c_{v,1}$ and $c_{v,2}$ on $E_v$.
 Then the tangent vectors of the two parametrizations are
 \begin{equation*}\begin{split}
 &(0, \ldots, 0, c_{v,1}^{m^{j_v'}_{v+1}}, \ldots , c_{v,1}^{m^{j_v''}_{v+1}}, 0,\dots, 0)\\
 &(0, \ldots, 0,c_{v,2}^{m^{j_v'}_{v+1}}, \ldots ,c_{v,2}^{m^{j_v''}_{v+1}}, 0, \dots, 0).
 \end{split}\end{equation*}
 Property (\ref{eq:diff}) guarantees that they are linearly independent whenever we have at least two non--trivial
 linear entry in each vector.
Recall that the number of non--trivial linear entries is the number of 1's in the column $v$, hence
  $k_v-\val_v+1$. Therefore, the tangent vectors of
   any   $k_v-\val_v+1$ different disc (transversal to $E_v$),  mapped by $\Psi$,  will
  be linearly independent in $\CC^{t+2}$. (The determinant of the system is a nonzero
  generalized Vandermonde determinant.)
This is true for any $v$.

On the other hand, if we wish to put discs simultaneously we need an additional care.
The  $k_v-\val_v+1$ tangents lines of $v$ and the  $k_{v+1}-\val_{v+1}+1$ tangent lines of $v+1$ generate a subspace
of dimension $(k_v-\val_v+1)+( k_{v+1}-\val_{v+1}+1)-1$. Note that
the $-1$ is responsible for  the common row  $j$  corresponding to $m^j_v=m^j_{v+1}=1$.

More generally,  if we place  $r_v$ discs  on each $E_v$, then we have to understand the rank of the following matrix.
Consider an $(s,t)$--matrix (corresponding to the $\{m^j_v\}_{1\leq v\leq s, 1\leq j\leq t}$ matrix).
Replace  each column  $v$ by $r_v$ columns, indexed by $i$, (hence replace the $v$ column by a
$(r_v,t)$--matrix $M_v$) whose entries corresponding to $(i,j)$, $j_v'\leq j\leq j_v''$,
 are  $ c_{v,i}^{m^{j}_{v+1}}$, where all $\{c_{v,i}\}_i$ are different. The other entries of $M_v$ are zero.
 This is the matrix of the tangent line of the curves $C$, the union of the images of the discs by $\Psi$.

 Eg., in the case of $[2,2,3,2,2]$ if we  start with $r_1=2 $ and $r_2=1$, then the left upper corner of the matrix is
 $$\begin{pmatrix}  c_{1,1}^2 & c_{1,2}^2 & 0 &  \ \\
 c_{1,1} & c_{1,2} & c_{2,1} & \ \\
 0&0&0& \ \\
 \ & \ & \ & \ \end{pmatrix}
$$
Hence the tangent directions generate a subspace of dimension two. Note that any combination $(r_1,r_2)=(2,0)$ or $(1,1) $
generate a subspace of dimension 2. More generally, if we wish linear independence of the tangent vectors, then the entries
$(r_1,r_2)$ must satisfy $r_v\leq k_v-\val_v+1$ but also $r_1+r_2\leq  (k_1-\val_1+1)+(k_2-\val_2 +1)-1$.
This discussion shows the following result.

\begin{theorem}
Assume that on each $E_v$ we put $r_v$ different transversal discs. Then the tangent vectors of their
images by $\Psi$ are linearly independent whenever we have the following
inequalities: for any $1\leq v_1\leq v_2\leq s$
\begin{equation}\label{eq:ri}
\sum _{v_1\leq v\leq v_2}r_v\leq \sum  _{v_1\leq v\leq v_2}( k_v-{\rm val}_v+1) -(v_2-v_1)=1+
\sum  _{v_1\leq v\leq v_2}( k_v-{\rm val}_v).\end{equation}
Furthermore, these curves in $\CC^{t+2}$ form a $(\sum_vr_v)$--tuple.
\end{theorem}
\noindent
The last statement follows by induction and using Hironaka's formula (\ref{eq:deltaA1}).

\begin{example} One can  verify that the curve
 $\pi(\widetilde{C}_{Z_{\min}}\cup \widetilde{D}^+) $ satisfies (\ref{eq:ri}).
 Indeed, $\widetilde{C}_{Z_{\min}}$ corresponds to $\sum_vr_v E^*_v$ with
 $r_v=k_v-\val_v $ for all $v$. The additional $+1$ on the right hand side of (\ref{eq:ri}) takes care of the additional $\widetilde{D}^+$ as well.
\end{example}

\begin{example} For any $0<a<n$ write $s_{a[E^*_s]}$ as $\sum_v a_v E^*_v$ as in subsection \ref{eq:sh}.
Then the entries $\{a_v\}_v$ satisfy  (\ref{eq:ri}) with $r_v=a_v$. This can be verified by induction
using the steps of the algorithm from \ref{eq:sh}.
First one verifies the starting case
 $(k_1-1, k_2-2, \ldots, k_s-2)$, then one follows the inductive step of the algorithm.
\end{example}

\begin{remark}\label{rem:utolso}
In the above discussions we used the minimal resolution. As we already mentioned, if we use any other resolution, along the  newly created
irreducible exceptional curves the multiplicities of $Z_{\min}$ will be $\geq 2$, hence these components cannot support strict transforms
of smooth curves. So, our treatment  in the minimal resolution in fact provides the complete  discussion.
\end{remark}

\bibliographystyle{amsplain}

\providecommand{\bysame}{\leavevmode\hbox to3em{\hrulefill}\thinspace}
\providecommand{\MR}{\relax\ifhmode\unskip\space\fi MR }
\providecommand{\MRhref}[2]{%
  \href{http://www.ams.org/mathscinet-getitem?mr=#1}{#2}
}
\providecommand{\href}[2]{#2}

\end{document}